\documentclass[11pt]{amsart}
\usepackage{verbatim, latexsym, amssymb, amsmath}
\usepackage{epsfig}
\def\R{\mathbb R}

\def\H{\mathcal H}

\def\ric{\mathrm{Ric}}
\def\area{\mathrm{area}}

\def\d{\mathrm{div}}

\newtheorem{thm}{Theorem}[section]
\newtheorem{lemm}[thm]{Lemma}
\newtheorem{cor}[thm]{Corollary}
\newtheorem{prop}[thm]{Proposition}
\theoremstyle{remark}
\newtheorem*{rmk}{Remark}

\theoremstyle{definition}

\title{Rigidity of min-max minimal spheres in three-manifolds}
\author{F. C. Marques and A. Neves}
\address{Instituto de Matem\'atica Pura e Aplicada (IMPA) \\ Estrada Dona Castorina 110 \\ 22460-320 Rio de Janeiro \\ Brazil}
\address{Imperial College London \\ Huxley Building \\ 180 Queen's Gate \\ London SW7 2RH \\ United Kingdom}
\thanks{The first author was supported by CNPq-Brazil, FAPERJ, and Math-Amsud. The second author was supported by Marie Curie IRG grant.}

\begin{document}

\begin{abstract}
In this paper we consider min-max minimal surfaces in three-manifolds and prove some rigidity results.
For instance,  we prove that any metric on a $3$-sphere which has scalar curvature greater than or equal to $6$  and is not round must have an embedded minimal sphere of area strictly smaller than $4\pi$ and index at most one.
If the Ricci curvature is positive we also prove sharp estimates for the width.
\end{abstract}

\maketitle

\section{Introduction}

Let $M$ be a compact Riemannian three-manifold. It is well-known that lower bounds on the scalar curvature of $M$  give some information on the space of minimal surfaces. Several
rigidity theorems have been obtained assuming the existence of an area-minimizing surface of some kind (\cite{bben}, \cite{bbn}, \cite{Cai-Galloway}, \cite{Eichmair}, \cite{nunes}, \cite{Schoen-Yau2}), but no known result asserts rigidity under the presence of a minimal surface produced by min-max methods. In this paper we prove theorems in that direction.

Let $g$ be a metric on the three-sphere $S^3$. Before we state our first theorem, we introduce the definition of width. We start with the family $\{\overline{\Sigma}_t\}$ of level sets of the height function $x_4:S^3 \subset \mathbb{R}^4 \rightarrow \mathbb{R}$, i.e., 
$$\overline{\Sigma}_t=\{ x \in S^3: x_4=t\}$$ 
for $t \in [-1,1]$. Then we define $\overline{\Lambda}$ to be the collection of all  families $\{\Sigma_t\}$ with the property that  $\Sigma_t=F_t(\overline{\Sigma}_t)$ for some smooth one-parameter family of
diffeomorphisms $F_t$ of $S^3$, all of which isotopic to the identity. The {\it width} of $(S^3,g)$ is the min-max invariant
$$
W(S^3,g)=\inf_{\{\Sigma_t\}\in\overline{\Lambda}}\, \sup_{t\in [-1,1]}|\Sigma_t|,
$$
where $|\Sigma|$ denotes the surface area of $\Sigma$.

We prove (Theorem \ref{thm.sphere}):
\begin{thm}\label{mainthm.1} Let $g$ be a metric of positive Ricci curvature on $S^3$, with scalar curvature $R \geq 6$.  There exists an embedded minimal sphere
$\Sigma$, of index one, such that
$$W(S^3,g)=|\Sigma| \leq 4\pi.$$
The  equality $W(S^3,g)=4\pi$ holds if and only if $g$  has constant sectional curvature one.
\end{thm}

The proof of the rigidity statement uses short-time existence for  Hamilton's Ricci flow and the Maximum Principle, in the same spirit as in  \cite{bben}. Connections between Ricci flow
and the theory of minimal surfaces have been previously considered in \cite{hamilton-survey} and \cite{Colding-Minicozzi1}.

\medskip

\begin{rmk}
 It is interesting to compare Theorem \ref{mainthm.1} with the rigidity proven by Llarull \cite{Llarull} (any distance increasing deformation of the round sphere decreases scalar curvature somewhere). 
\end{rmk}
\medskip

Simon and Smith \cite{smith} proved that any metric on $S^3$  admits an embedded minimal sphere (see \cite{Jost} and \cite{white2} for related results). For  metrics of positive scalar curvature, we have (Theorem \ref{thm.sphere.scalar.curvature}):
\begin{thm}\label{mainthm.2}
Let $g$ be a metric  on $S^3$ with scalar curvature $R \geq 6$. If $g$ does not have constant sectional curvature one, then there exists an embedded minimal sphere $\Sigma$, of index zero or one, with  $|\Sigma| < 4\pi$.
\end{thm}

\begin{rmk} Toponogov \cite{Toponogov} proved that the length of any simple closed geodesic contained in a two-sphere with scalar curvature $R \geq 2$ is at most $2\pi$. See \cite{Hang-Wang2} for a different proof. It is not difficult to see that there is no such bound for the area of a minimal sphere (consider a long piece of a cylinder around the $x_1$ axis in $\mathbb{R}^4$, capped at both ends in such a way that $R\geq 6$ and  the intersection with the $\{x_2=0\}$ hyperplane is a minimal sphere).  Thus Theorem \ref{mainthm.2} can be seen as 
a three-dimensional version of Toponogov Theorem.
\end{rmk}

\medskip

The search for scalar curvature rigidity results for the standard sphere is motivated by the Positive Mass Theorems in both the Euclidean (\cite{Schoen-Yau1}, \cite{Witten}) and the hyperbolic (\cite{Chrusciel-Herzlich}, \cite{Wang}, see also \cite{Andersson-Dahl}, \cite{Min-Oo1}) settings. In \cite{Brendle-Marques-Neves}, Brendle, Marques and Neves constructed smooth metrics on the hemisphere $S^n_+$, if $n\geq 3$,  that coincide with the standard
metric near the equator $\partial S^n_+$, the scalar curvature satisfies $R \geq n(n-1)$ everywhere, and $R>n(n-1)$ somewhere. These metrics are counterexamples to the so-called Min-Oo's Conjecture \cite{Min-Oo2}, which was a natural proposal for a Positive Mass Theorem in the spherical setting.   The theorems we prove in this paper are reminiscent of this conjecture.

When the ambient manifold is not a sphere, we prove some existence and rigidity theorems assuming the Ricci curvature is positive  (Theorem \ref{ricci.rigidity}): 
\begin{thm}\label{mainthm.3}
Let $g$ be a metric of positive Ricci curvature on $M=S^3/\Gamma$, with scalar curvature $R \geq 6$. If $(M,g)$ is not isometric to the standard sphere, then there exists an embedded minimal surface $\Sigma$, of index zero or one, with  $|\Sigma| < 4\pi$.
\end{thm}

The surface $\Sigma$ is either   non-orientable with $|\Sigma|\leq 2\pi$ or orientable and    constructed by min-max methods. The proof of Theorem \ref{mainthm.3}, including the area estimate, is based on Hamilton's convergence theorem \cite{hamilton} for  the Ricci flow of three-manifolds with positive Ricci curvature. Nonetheless, we conjecture that the assumption of positive Ricci curvature in Theorem \ref{mainthm.3} is not necessary.

Finally, it follows from the work of Pitts \cite{pitts} that any compact Riemannian three-manifold $M$ admits an embedded  minimal surface (compact). The techniques of this paper give some extra geometric information on the surface. If $M$ is orientable, for instance, we prove that the minimal surface can be chosen to have index less  than or equal to one (Corollary \ref{consequence}).

\medskip

{\bf Acknowledgements:} The authors thank Hossein Namazi for many useful discussions.


\section{Min-max Minimal surfaces}

 Let $M$ be a compact Riemannian three-manifold, possibly with boundary. 
We begin with some  definitions. The 2-dimensional Hausdorff measure of $\Sigma \subset M$ will be denoted by $\mathcal{H}^2(\Sigma)$. If $\Sigma$ is
a surface, then $\mathcal{H}^2(\Sigma)$ equals the area  $|\Sigma|$ of $\Sigma$. The surfaces in this paper will be assumed to be connected and closed, unless otherwise indicated.

Let $I = [a,b] \subset \mathbb{R}$ be a closed interval. Let $\{\Sigma_t\}_{t \in I}$ be a family of closed subsets of $M$ with finite $\mathcal{H}^2$-measure such that 
\begin{enumerate}
\item[(c1)] $\mathcal{H}^2(\Sigma_t)$ is a continuous function of $t \in I$, 
\item[(c2)] $\Sigma_t$ converges to $\Sigma_{t_0}$, in the Hausdorff topology, as $t \rightarrow t_0$.
\end{enumerate}

We say that $\{\Sigma_t\}$ is a  {\it generalized family of surfaces} (or a {\it sweepout}) if there are finite sets $T\subset I$ and $P \subset M$ such that:
\begin{enumerate}
\item[(a)] if $t \in I \setminus T$ then $\Sigma_t$ is a surface in $M$, 
\item[(b)] if $t \in T$ then either $\Sigma_t \setminus P$ is a surface in $M$  or else $\H^2(\Sigma_t)=0$,
\item[(c)] $\Sigma_t$ varies smoothly in $[0,1]\setminus T$,
\item[(d)] if $t\in T$ and $\H^2(\Sigma_t)\neq 0$, then $\Sigma_{\tau}$ converges smoothly to $\Sigma_t$ in $M\setminus P$ as $\tau \rightarrow t$.
\end{enumerate}

\begin{rmk}
Conditions (b) and (d) are stated slightly differently from the corresponding conditions in  \cite{delellis-genus} in order to allow  $\Sigma_t$ to be a graph for some $t\in T$. Because  the area of such $\Sigma_t$ is zero, all  results in \cite{delellis-genus} carry through without having to modify their proofs.
\end{rmk}


Let  $\Lambda$ be a collection of generalized families of surfaces. We denote by $\mbox{Diff}_0$ the set of diffeomorphisms of $M$ which are isotopic to the identity map. If $\partial M \neq \emptyset$  we require the isotopies to leave some neighborhood of $\partial M$ fixed.

The set $\Lambda$  is  {\em saturated} if given a map $\psi \in C^{\infty}(I\times M, M)$ such that $\psi(t,\cdot) \in \mbox{Diff}_0$ for all $t\in I$, and a family  $\{\Sigma_t\}_{t\in I}\in \Lambda$, we have $\{\psi(t,\cdot)(\Sigma_t)\}_{t \in I}\in \Lambda$. We require also  the existence of $N_0=N_0(\Lambda)>0$ such that the set $P$ has at most $N_0$ points for any $\{\Sigma_t\}_{t\in I}\in \Lambda$.

The {\it width of $M$ associated with $\Lambda$} is defined to be
$$W(M,\Lambda)=\inf_{\{\Sigma_t\}\in\Lambda}\sup_{t\in I}\H^2(\Sigma_t).$$

Now suppose that $\partial M \neq \emptyset$. We denote the mean curvature of the boundary by $H(\partial M)$. Here the convention is that the mean
curvature vector is $-H(\partial M) \nu$, where $\nu$ is the outward unit normal. 

In this case we choose $I=[0,1]$ and require the extra condition that any $\{\Sigma_t\}_{t \in [0,1]} \in \Lambda$ satisfies
\begin{enumerate}
\item[(c3)] $\Sigma_0 = \partial M$, $\Sigma_t \subset {\rm int}(M)$ for $t>0$, and $\{\Sigma_t\}$ foliates a neighborhood of $\partial M$.  This last condition  means that there exists a smooth function $w : [0, \varepsilon_0] \times \partial M \rightarrow \mathbb{R}$, satisfying $w(0,x)=0$ and $\frac{\partial w}{\partial t}(0,x)>0$, such that
$$
\Sigma_t = \{ \exp_x(-w(t,x)\nu(x)): x \in \partial M\}
$$
for any $t \in [0,\varepsilon_0]$.
\end{enumerate}

The goal of this section is to prove:

 \begin{thm}\label{boundary.min-max} Let $(M,g)$ be a compact three-manifold with boundary such that  $H(\partial M)>0$. For any saturated set   $\Lambda$ with 
$W(M,\Lambda)>|\partial M|,$
 there exists  a min-max sequence obtained from $\Lambda$ that converges in the varifold sense to an embedded minimal surface $\Sigma$ (possibly disconnected) contained in the interior of $M$. The area of $\Sigma$ is equal to $W(M,\Lambda)$, if counted with multiplicities.
 \end{thm}
\begin{proof}

A {\it minimizing sequence} is a sequence of families $\{\Sigma_t^n\} \in \Lambda$ such that 
$$
\lim_{n \rightarrow \infty} \sup_{t\in I}\H^2(\Sigma^n_t)= W(M,\Lambda).
$$
A {\it min-max sequence} is then a sequence of slices  $\Sigma_{t_n}^n$, $t_n \in I$, such that $\mathcal{H}^2(\Sigma_{t_n}^n) \rightarrow W(M,\Lambda)$
as $n \rightarrow \infty$.

It is enough to show that we can find a minimizing sequence $\{\Sigma_t^n\} \in \Lambda$, $a>0$, and $\delta>0$ such that 
  $$
  |\Sigma_t^n| \geq  W(M,\Lambda)-\delta \Rightarrow d(\Sigma_t^n, \partial M) \geq a/2,
  $$
  because the theorem then follows from simple modifications of the arguments presented in \cite{colding-delellis}.

In a  neighborhood of $\partial M$, the metric can be written as $g=dr^2+g_{r}$ on  $[0,2a]\times \partial M$   for some $a>0$, where $\partial M$ is identified with $\{0\}\times\partial M$.  If  $H(\partial M)>0$, we can choose $a$ sufficiently small so that the mean curvature of $C_r=\{r\}\times \partial M$ is positive 
 for every $r \in [0,2a]$. We denote by $M_r$ the complement of $[0,r) \times \partial M$, and by $A$ the second fundamental form  of $C_r$.

\begin{lemm}\label{new.family}
 For any $\{\Sigma_t\}\in \Lambda$ and any $t_0 \in (0,1)$, there exists  a smooth one-parameter family of diffeomorphisms  $(F_t)_{0\leq t\leq 1}$  of $M$ so that
  \begin{itemize}
  \item $F_0 = {\rm id}$,
  \item $F_t={\rm id}$  in a  neighborhood $U$ of $\partial M$,
  \item $|F_t(\Sigma_t)| \leq |\Sigma_t|$,
  \item for any $t \geq t_0$,  we have $F_t(\Sigma_t) \subset M_{a/2}$.
  \end{itemize}
\end{lemm}

\begin{proof}
Let $\{\Sigma_t\}\in \Lambda$ and  $t_0 \in (0,1)$. Choose $\eta >0$ sufficiently small so that $\eta \leq a/8$ and $d(\Sigma_t, \partial M) \geq 2\eta$ for all $t \in [t_0/2,1]$.  

If  $c=\sup_{x\in C_r, 0\leq r\leq 2a}|A|$, we choose  a non-negative real function $\phi$ so that $\phi'\leq-c\phi$, $\phi(r)>0$ for $r<a$, and $\phi(r)=0$ for $r\geq a$. For instance, take $\alpha$ to be a nonnegative and non-increasing function with  $\alpha(r)>0$ for $r<a$, $\alpha(r)=0$ for $r\geq a$, and set $\phi(r)=\alpha(r)\exp(-cr)$. Now choose $\kappa$ a nonnegative function such that $\kappa(r)=0$ for $r \leq \eta$ and $\kappa(r)=1$ for $r \geq 2\eta$. 
	
	Denote by $(\tilde{F}_t)_{0\leq t<\infty}$ the one-parameter family of diffeomorphisms generated by the vector field $X=\kappa(r) \phi(r)\frac{\partial}{\partial r}$. 

\medskip
	
{\bf Claim.} For every surface $L\subset M_{2\eta}$, the function
$t \rightarrow \area(\tilde{F}_t(L))$ is non-increasing. In particular, $|\tilde F_t(L)| \leq |L|$ if $t \geq 0$.

\medskip
	
	We have
	$$\frac{d}{dt}\area(\tilde F_t(L))=\int_{\tilde F_t(L)}\d_{\tilde F_t(L)} Xd\mu.$$
Thus it suffices to check that for every orthonormal basis $\{e_1,e_2\}$ we have $\sum_{i=1}^2\langle \nabla_{e_i}X,e_i\rangle\leq 0$. Notice that $\kappa \equiv 1$ in $M_{2\eta}$ and $\tilde F_t(L)\subset M_{2\eta}$ for all $t\geq 0$. Without loss of generality we can assume that $e_1$ is orthogonal to $\frac{\partial}{\partial r}$. We denote by $e_1^*$ a unit vector tangent to $C_r$ and orthogonal to $e_1$. Direct computation shows that $\nabla_{\frac{\partial}{\partial r}}\frac{\partial}{\partial r}=0$ and thus, denoting by $\pi$ the projection of a tangent vector in $M$ into the tangent space at $C_r$, we get
\begin{eqnarray*}
	\sum_{i=1}^2\langle \nabla_{e_i}X,e_i\rangle&=&\phi'\left\langle e_2, \frac{\partial}{\partial r}\right\rangle^2+\phi\sum_{i=1}^2\left \langle \nabla_{e_i}\frac{\partial}{\partial r},e_i\right\rangle\\
	&=&\phi'\left\langle e_2, \frac{\partial}{\partial r}\right\rangle^2-\phi\sum_{i=1}^2A(\pi(e_i),\pi(e_i))\\
	&=& \Big(\phi' + A(e_1^*,e_1^*)\phi\Big)\left\langle e_2, \frac{\partial}{\partial r}\right\rangle^2-\phi H\\
	&\leq& (\phi'+c\phi)\left\langle e_2, \frac{\partial}{\partial r}\right\rangle^2-\phi H\\
	&\leq& 0.
\end{eqnarray*}
This proves the claim.

\medskip

Notice  that $\tilde{F}_t$ is the identity in $M_a$ and $\lim_{t\to\infty} \tilde{F}_t(r,x)=(a,x) $ for all $x\in \partial M$ and $2\eta \leq r<a$. Let $T>0$ be such that $\tilde F_T(C_{2\eta}) = C_{a/2}$. We choose a smooth nonnegative function $h:[0,1] \rightarrow \mathbb{R}$
such that  $h(t)=0$ for $t \leq t_0/2$ and $h(t)=T$ for $t \geq t_0$. 

Define $F_t = \tilde F_{h(t)}$. Hence $F_0=\tilde F_0 = {\rm id}$. The second item of the lemma follows because $X=0$ outside $M_\eta$. In order to prove the third item we recall that if $t \geq t_0/2$ then $\Sigma_t \subset M_{2\eta}$. Hence it follows from the claim that $|F_t(\Sigma_t)|\leq |\Sigma_t|$. If $t \leq t_0/2$ the inequality is trivial since we have that  $h(t)=0$ and $F_t= {\rm id}$. Finally, if $t \geq t_0$ we have $F_t=\tilde F_T$. In that case, since $\Sigma_t \subset M_{2\eta}$, we conclude that
$F_t(\Sigma_t) \subset \tilde F_T (M_{2\eta}) = M_{a/2}$. This finishes the proof of the fourth item and of the lemma.
\end{proof}

We can now finish the argument.
Let $m_0=W(M,\Lambda)$ and choose $0< \delta < \frac12(m_0 - |\partial M|)$.

Let $\{\Sigma_t\}\in \Lambda$. There exists $\varepsilon>0$ such that the map $\Psi: [0,2\varepsilon] \times \partial M \rightarrow M$ given by $\Psi(t,x) =  \exp_x(-w(t,x)\nu(x))$ is a diffeomorphism onto a neighborhood of $\partial M$. We can choose $\varepsilon$ sufficiently small so that $|\Sigma_t| \leq |\partial M| + \delta$ for $t \in [0,2\varepsilon]$. 
By choosing $t_0=\varepsilon$ in Lemma \ref{new.family}, we obtain a family of generalized surfaces $\{\Sigma_t'\}\in \Lambda$, given by $\Sigma_t' = F_t(\Sigma_t)$, such that
\begin{itemize}
  \item $\sup_{t\in [0,1]}\H^2(\Sigma_t') \leq \sup_{t\in [0,1]}\H^2(\Sigma_t)$,
   \item if $|\Sigma_t'| \geq m_0-\delta$, then $\Sigma_t' \subset M_{a/2}$.
  \end{itemize}
  
  This means that we can restrict to minimizing sequences $\{\Sigma_t^n\} \in \Lambda$ such that 
  $$
  |\Sigma_t^n| \geq m_0-\delta \Rightarrow d(\Sigma_t^n, \partial M) \geq a/2,
  $$
and this concludes the proof.
\end{proof}


\section{Genus and Index}\label{section.minimal.index}

 From this point on $(M,g)$ will denote a compact orientable Riemannian three-manifold without boundary.  Furthermore, the surfaces will be assumed to be connected and closed, unless otherwise indicated.
The genus of a surface $\Sigma' \subset M$ will be indicated by $g(\Sigma')$.  Let $\Sigma \subset M$ be an embedded minimal surface. If $\Sigma$ is   orientable, the  {\it Jacobi operator} is given by
$$L\phi=\Delta\phi +|A|^2\phi+\ric(\nu,\nu)\phi,$$
where $\phi \in C^{\infty}(\Sigma)$, $A$ is the second fundamental form, and $\nu$ is a unit normal vector. If $\Sigma$ is non-orientable, we need to pass to the double cover $\tilde{\Sigma}$ and restrict $L$ to the functions $\phi \in C^\infty(\tilde \Sigma)$ that satisfy $\phi \circ \tau = - \phi$, where $\tau:\tilde \Sigma \rightarrow \tilde \Sigma$ is the orientation-reversing involution such that $\Sigma=\tilde{\Sigma}/\{id,\tau\}$. The {\it index} of $\Sigma$, denoted by  $\mbox{ind}(\Sigma)$, is the number of negative eigenvalues of $L$, counted with multiplicity. The surface $\Sigma$ is called {\it stable} if ${\rm ind}(\Sigma) \geq 0$. 

The next result will be useful in proving that certain min-max minimal surfaces have index one. Throughout the rest of the paper, we will only consider saturated sets $\Lambda$ such that no sweepout $\{\Sigma_t\}$ in $\Lambda$ contains a non-orientable surface. 

\begin{prop}\label{sweepout.aux}
Let $\Lambda$ be  a saturated set of generalized families of surfaces, and 
let $(\Sigma_t)_{-1\leq t\leq 1} \in \Lambda$ be a sweepout such that
 \begin{itemize}
 \item[(a)] $|\Sigma_t|<|\Sigma_0|$ for all $t\neq 0$,
 \item[(b)] $(\Sigma_t)_{-1\leq t\leq 1}$ is smooth around $t=0$, 
 \item[(c)] the function $f(t)=|\Sigma_t|$ satisfies $f''(0)<0$.
 \end{itemize}
 If 
 $$|\Sigma_0| = W(M,\Lambda,g),$$ then $\Sigma_0$ is an embedded minimal surface of index one.
\end{prop}

\begin{proof}

We argue first that $\Sigma_0$ must be a minimal surface. If not, we can consider an ambient vector field $X$  vanishing outside a tubular neighborhood of $\Sigma_0$ and identical to the mean curvature vector  on $\Sigma_0$. Denote by $(F_s)_{s\in \R}$ the one parameter family of diffeomorphisms generated by $X$ and set $f(t,s)=|F_s(\Sigma_t)|$. We have $\frac{df}{dt}(0,0)=0$ and $\frac{df}{ds}(0,0)<0$. Because $f$ has a unique global maximum at the origin we can find $\delta$ small and positive so that $f(t,\delta)<f(0,0)$ for all $t$, which is a contradiction because  the sweepout $(F_{\delta}(\Sigma_t))_{-1\leq t\leq 1}$ is in $\Lambda$.

 It remains to show that $\Sigma_0$ has index one. Notice that the condition (c) implies ${\rm ind}(\Sigma_0) \geq 1$.

Choose  a unit normal vector field $\nu$, along $\Sigma_0$, and let $\phi_0 \in C^{\infty}(\Sigma_0)$ be such  that the deformation vector of $\Sigma_t$ when $t=0$ is $Z=\phi_0\nu$. If the index of  $\Sigma_0$ is bigger than one, we can choose  orthogonal eigenfunctions  $\phi_1,\phi_2\in C^{\infty}(\Sigma_0)$   for the Jacobi operator $L$ with negative eigenvalues. There exists a  linear combination of $\phi_1$ and $\phi_2$, say $\phi_3$, so that
\begin{equation}\label{orthogonal}
\int_{\Sigma_0}\phi_3 L \phi_0d\mu=0\mbox{ and }\phi_3\neq 0. 
\end{equation}
Consider the normal vector field $\tilde X=\phi_3\nu$  and extend it smoothly to be zero outside a small tubular neighborhood of $\Sigma_0$. Denote by $(\tilde F_s)_{s\in \R}$ the one parameter family of diffeomorphisms generated by $\tilde X$ and set $\tilde f(t,s)=|\tilde F_s(\Sigma_t)|$. We have $\nabla \tilde f(0,0)=0$ (by minimality), $\frac{\partial^2}{\partial s\partial t} \tilde f(0,0)=0$ (by \eqref{orthogonal}), $\frac{\partial^2}{\partial t\partial t} \tilde f(0,0)<0$ (by assumption), and $\frac{\partial^2}{\partial s\partial s} \tilde f(0,0)<0$ (by the choice of $\phi_3$).  From the Taylor expansion of $\tilde f$ around $(0,0)$   and the fact that  $\tilde f$ has a unique global maximum at the origin, we can find $\delta$ small and positive so that $\tilde f(t,\delta)<\tilde f(0,0)$ for all $t$, which is a contradiction.

\end{proof}

An orientable surface $\Sigma$ is a {\em Heegaard splitting} if $M\setminus\Sigma$ has two connected components that are both  handlebodies, i.e., diffeomorphic to a solid ball with handles attached. The {\it Heegaard genus} of $M$ is the lowest possible genus of a Heegaard splitting of $M$.

 Given an integer $h \geq 0$, we denote by $\mathcal{E}_h$ the collection of  all connected embedded minimal surfaces $\Sigma \subset M$ with $g(\Sigma)\leq h$.
It is said that  $(M,g)$ satisfies the $(\star)_h$ {\it -condition}  if 
\begin{itemize}
\item $M$ does not contain embedded non-orientable surfaces,
\item  no surface in $\mathcal{E}_h$ is stable.
\end{itemize}

\begin{rmk} If $(M,g)$ has positive Ricci curvature and  does not contain embedded non-orientable surfaces, then $M$ satisfies the $(\star)_h$-condition for all $h$. Lens spaces $L(p,q)$
with odd $p$ (see \cite{Bredon-Wood}) and the Poincar\'{e} homology sphere are some examples.
\end{rmk}

   \begin{lemm}\label{handlebody}
  Assume $(M,g)$ satisfies the $(\star)_h$ -condition. Then any surface $\Sigma \in \mathcal{E}_h$ is a Heegaard splitting.  \end{lemm}
The version we state here was essentially proven in \cite{msy} (see also \cite{lawson}).
 \begin{proof}
 Let $\Sigma \in \mathcal{E}_h$. We   first argue that $\Sigma$ must separate. Because $M$ is orientable and $\Sigma$ is connected, $M\setminus\Sigma$ consists of one or two connected components. In the first case, we can choose $\phi\in C^{\infty}(\Sigma)$ an eigenfunction for the lowest eigenvalue $\lambda$ of the Jacobi operator $L$. Note that $\lambda <0$,  because  $\Sigma$ is unstable. Consider a vector field $X$ in  $M$ such that $X=\phi \nu$ on $\Sigma$, where $\nu$ is a unit normal vector to $\Sigma$, and denote by $(F_t)_{t\in \R}$ the flow generated by  $X$. Because $\phi$ can be taken to be strictly positive and $\Sigma$ does not separate we have that, for all $t$  sufficiently small,
$$M\setminus(F_{t}(\Sigma)\cup F_{-t}(\Sigma)) =A_t\cup B_t$$
 where $A_t, B_t$ are two disjoint and connected open regions  with $\Sigma \subset B_t$.
Moreover, 
 \begin{equation*}
 \frac{\partial}{\partial t}\langle \vec{H}(F_t(\Sigma)),\nu_t\rangle _{|t=0}=L\phi=-\lambda\phi>0,
\end{equation*}
where $\nu_t$ denotes the unit normal vector to $F_t(\Sigma)$. Hence we have that for sufficiently small $t$ the mean curvature vector of $\partial A_{t}$  points into $A_t$, i.e., $\partial A_t$ is mean convex.  
 Thus  we can minimize area in the isotopy class of one of the boundary components of $\partial A_t$, as in \cite{msy} (see also \cite{Hass-Scott}), to obtain an embedded stable minimal surface in $A_t$. The genus of this surface is at most $h$, contradicting the definition of the $(\star)_h$-condition. Therefore $M\setminus \Sigma$ is a union of two connected components.
  
 It remains to prove that each connected component  is a handlebody. Let $N$ be such a component.  If $N$ is not a handlebody, and since  $\Sigma=\partial N$ is an unstable minimal surface, we  can minimize area in its isotopy class  to obtain a stable minimal surface in the interior of $N$   with genus less than or equal to $h$ (see Proposition 1 of \cite{msy} for a characterization of handlebodies). This again violates the $(\star)_h$-condition, and finishes the proof of the lemma.
  \end{proof}

 \begin{lemm}\label{lemm.no.stable}  If $(M,g)$ satisfies the $(\star)_h$ -condition, then any $\Sigma \in \mathcal{E}_h$   must intersect every other embedded minimal surface.
 \end{lemm}
 \begin{proof} 
 Suppose  that $\Sigma_1$ and $\Sigma_2$ are disjoint  embedded minimal surfaces, with $\Sigma_1 \in \mathcal{E}_h$.  Because $\Sigma_1$ is a Heegaard splitting,  there is a region $B$ of $M$, homeomorphic to a handlebody, such that $\Sigma_2\subset B$ and $\Sigma_1=\partial B$. Therefore there is a region $C$ of $M$ such that $\partial C=\Sigma_1 \cup \Sigma_2$. It follows from the results of \cite{msy} that we can minimize area in the isotopy class of $\Sigma_1$ to obtain an embedded stable minimal surface of genus less than or equal to $h$ in $C$. This is in contradiction with the assumption that $(M,g)$ satisfies the $(\star)_h$-condition. The lemma follows. 
  \end{proof}

 If $\Sigma$ is a Heegaard splitting of $M$, then there is a natural class of sweepouts $(\Sigma_t)_{-1\leq t\leq 1}$ we can associate to $\Sigma$.  Each $(\Sigma_t)_{-1\leq t\leq 1}$ satisfies
 \begin{itemize}
 \item $\Sigma_0=\Sigma$ and $\Sigma_t$ is an embedded surface isotopic to $\Sigma$ for all $-1<t<1$;
 \item  if $N_1, N_2$ denote the connected components of $M\setminus \Sigma$ then $(\Sigma_t)_{-1\leq t\leq 0}$ (or  $(\Sigma_t)_{0\leq t\leq 1}$)  foliate $N_1$ (or $N_2$) with $\Sigma_{-1}$ (or $\Sigma_{1}$) being a graph. 
 \end{itemize}
 This class generates a saturated set $\Lambda_\Sigma$ of generalized families of surfaces which we call  the {\it  saturated set associated with} $\Sigma$. If the genus of $\Sigma$ is $h$, the {\em large saturated set associated with} $\Sigma$, denoted by  $\Lambda^h$, is defined to be the union of all saturated sets associated with Heegaard splittings of genus $h$.

   \begin{thm}\label{heegaard.genus} Suppose $(M,g)$  satisfies the $(\star)_h$ -condition, where  $h$ is the Heegaard genus of $M$. There exists an   orientable embedded minimal surface $\Sigma_0$, of genus $h$, such that  $$|\Sigma_0| = \inf_{S \in \mathcal{E}_{h} }|S|=W(M,\Lambda_{\Sigma_0})=W(M,\Lambda^h).$$ 
   Moreover, $\Sigma_0$ has index one and  is contained in a sweepout
 $(\Sigma_t)_{-1\leq t\leq 1} \in \Lambda_{\Sigma_0}$  with
 \begin{itemize}
 \item[(a)] $|\Sigma_t|<|\Sigma_0|$ for all $t\neq 0$,
 \item[(b)] $(\Sigma_t)_{-1\leq t\leq 1}$ is smooth around $t=0$, 
 \item[(c)] the function $f(t)=|\Sigma_t|$ satisfies $f''(0)<0$.
 \end{itemize}
  \end{thm}
  
\begin{proof}
From \cite{pitts} we know  that $(M,g)$ admits at least one  embedded minimal surface. Therefore  we can consider $h'$ to be the lowest possible genus of an embedded minimal surface in $M$.  Standard compactness of minimal surfaces  (for instance, Theorem 3 of \cite{white} or Theorem 4.2 of \cite{anderson}) implies  we can find a minimal surface $\Sigma_0$ of genus $h'$ which has smallest area among all embedded minimal surfaces with genus $h'$ . Hence
$$|\Sigma_0| = \inf_{S \in \mathcal{E}_{h'} }|S|.$$

We now argue that $h'=h$. Let $\Sigma$ be a Heegaard splitting of $M$ which has lowest possible genus $h$, and consider $\Lambda$ the saturated set associated with $\Sigma$. From  Theorem 0.6 of \cite{delellis-genus}  we know that $W(M,\Lambda,g)$ is achieved by a disjoint union of embedded minimal surfaces of genus at most $h$ with possible multiplicities  (there are no non-orientable minimal surfaces). Because $M$ satisfies the $(\star)_h$ -condition, there can be only one component $\Sigma^*$ by Lemma \ref{lemm.no.stable}. Note that $\Sigma^*$ must be  orientable. Hence    $h' \leq g(\Sigma^*) \leq h$. Since $g(\Sigma_0)=h'$, it follows from  Proposition \ref{handlebody} that $\Sigma_0$ is a Heegaard splitting. Therefore we also have $h\leq h'$. Hence $h=h'$.

Now:
 \begin{lemm}\label{prop.index} 
$$|\Sigma_0|=W(M,\Lambda_{\Sigma_0})=W(M,\Lambda^h).$$ Moreover, $\Sigma_0$ has index one and  is contained in a sweepout
 $(\Sigma_t)_{-1\leq t\leq 1} \in \Lambda_{\Sigma_0}$  such that
 \begin{itemize}
 \item[(a)] $|\Sigma_t|<|\Sigma_0|$ for all $t\neq 0$,
 \item[(b)] $(\Sigma_t)_{-1\leq t\leq 1}$ is smooth around $t=0$, 
 \item[(c)] the function $f(t)=|\Sigma_t|$ satisfies $f''(0)<0$.
 \end{itemize}
 \end{lemm}
 
  \begin{proof}
 From Lemma \ref{handlebody} we know that  $M\setminus \Sigma_0$ is the disjoint union of two handlebodies: $N_1$ and $N_2$.
 Set $\phi\in C^{\infty}(\Sigma_0)$ to be an eigenfunction for the lowest eigenvalue $\lambda<0$ of the Jacobi operator $L$,  and consider a vector field $X$ in  $M$ such that $X=\phi \nu$ on $\Sigma_0$, where $\nu$ is a unit normal vector pointing into $N_1$. Denote by $(F_t)_{t\in \R}$ the flow generated by $X$. Because $\phi$ can be taken to be strictly positive we have  
 \begin{equation}\label{nondegenerate}
\frac{\partial}{\partial t} \langle \vec H(\Sigma_t),\nu\rangle_{|t=0}=L\phi=-\lambda\phi>0.
\end{equation} Hence there exists $\varepsilon>0$ small enough so that
\begin{itemize}
 \item $\Sigma_t=F_t(\Sigma)$ is contained in $N_1$ (in $N_2$) for all $0<t<\varepsilon$  (for all $-\varepsilon<t<0$), 
 \item the mean curvature vector $H(\Sigma_t)$ points into $N_1$ (into $N_2$) for all $0<t<\varepsilon$  (for all $-\varepsilon<t<0$), 
 \item $|\Sigma_t|<|\Sigma|$ for all $0<|t|<\varepsilon$.
 \end{itemize}
 The surface $\Sigma_{\varepsilon}$ bounds a handlebody $N$ and so we can consider a sweepout 
 $\{\tilde{\Sigma}_t\}_{t \in [0,1]}$ of $N$ such that $\tilde{\Sigma}_t = \Sigma_{t+\varepsilon}$ for small $t$. Denote the corresponding set of saturated families of $N$ by $\tilde \Lambda$. If $W(N,\tilde \Lambda)>|\partial N|$ we can apply Theorem \ref{boundary.min-max} (because $H(\partial N)>0$) and derive the existence of a minimal surface $\Sigma_1$ in the interior of $N$, and thus disjoint from $\Sigma_0$. This contradicts  Lemma \ref{lemm.no.stable}.  Since  $|\partial N|<|\Sigma|$, we can find $\{\tilde \Sigma_t\} \in \tilde \Lambda$ so that
 $$\sup_{0\leq t\leq 1}|\tilde \Sigma_t|<|\Sigma|.$$
 Arguing in the same way for $\Sigma_{-\varepsilon}$ we conclude the existence of a sweepout $(\Sigma_t)_{-1\leq t\leq 1}$ of $M$ satisfying properties (a) and (b). Property (c) follows because from \eqref{nondegenerate} we obtain
 $$f''(0)=-\int_{\Sigma_0}\phi L\phi d\mu=\lambda\int_{\Sigma_0}\phi^2d\mu<0.$$

 Suppose that $W(M,\Lambda_{\Sigma_0},g)<|\Sigma_0|$. Then, by Theorem 0.6 of  \cite{delellis-genus}, there is an embedded minimal surface $S$ with $|S|<|\Sigma_0|$ and  $g(S)\leq g(\Sigma_0)$
  ($S$ has to be orientable).  This is a contradiction since  $\Sigma_0$ has least area in $\mathcal{E}_h$. Hence 
 $|\Sigma_0|=W(M,\Lambda_{\Sigma_0})$, and similarly we prove $|\Sigma_0|=W(M,\Lambda^h)$. The fact that $\Sigma_0$ has index one follows from  Proposition \ref{sweepout.aux}.
 \end{proof}
The statement of the theorem follows directly from the previous lemma.
\end{proof}

If  $M$  contains  a non-orientable embedded surface,   set $\tilde h$ to be the lowest genus of a non-orientable surface embedded in $M$, i.e., there is an embedding of $N_{\tilde h}$ (the non-orientable surface with genus $\tilde h$) into $M$ and every non-orientable embedded  surface in $M$ has genus greater than or equal to  $\tilde h$.

We denote by $\mathcal{F}$ the set of all embedded surfaces $\Sigma$ in $M$   homeomorphic to $N_{\tilde h}$. 
From Lemma 1 of  \cite{msy} we have that non-orientable  surfaces can not have arbitrarily small area. 

We define
$$\mathcal{A}(M,g)=\inf\{|S|\,|\, S\in \mathcal F\}.$$
 The next proposition follows from simple modifications of Proposition 5 in  \cite{bben}. We include the proof   for the
sake of completeness.

\begin{prop}\label{nonorientable} For every metric $g$ on $M$ there exists an embedded stable minimal surface $\Sigma\in \mathcal F$ with 
$|\Sigma|=\mathcal{A}(M,g)$.
\end{prop}

\begin{proof}
   We can find a sequence of surfaces $\Sigma_k \in \mathcal{F}$ such that 
\[|\Sigma_k| \leq \mathcal{A}(M,g) + \varepsilon_k,\] 
where $\varepsilon_k \to 0$ as $k \to \infty$. This implies 
\[|\Sigma_k| \leq \inf_{\Sigma \in \mathcal{J}(\Sigma_k)} |\Sigma| + \varepsilon_k,\] 
where $\mathcal{J}(\Sigma_k)$ denotes the collection of all embedded surfaces isotopic to $\Sigma_k$. By 
Theorem 1 of \cite{msy}, a subsequence of the sequence $\Sigma_k$ converges weakly to a disjoint union of smooth embedded minimal surfaces $\Sigma^{(1)}, \hdots, \Sigma^{(R)}$ with positive integer multiplicities and,  in particular, we have 
\begin{equation} 
\label{area}
\sum_{j=1}^R n_j \, |\Sigma^{(j)}| \leq \mathcal{A}(M,g). 
\end{equation}
 We define surfaces $S_k^{(1)}, \hdots, S_k^{(R)}$ as follows: if $n_j = 2m_j$ is even, then $S_k^{(j)}$ is defined by 
\[S_k^{(j)} = \bigcup_{r=1}^{m_j} \Big \{ x \in M: d(x,\Sigma^{(j)}) = \frac{r}{k} \Big \}\] 
On the other hand, if $n_j = 2m_j + 1$ is odd, then $S_k^{(j)}$ is defined by 
\[S_k^{(j)} = \Sigma^{(j)} \cup \bigcup_{r=1}^{m_j} \Big \{ x \in M: d(x,\Sigma^{(j)}) = \frac{r}{k} \Big \}.\] 
By  Remark 3.17 of \cite{msy}, we can find embedded surfaces $S_k^{(0)}$ and $\tilde{\Sigma}_k$ with the following properties: 
\begin{itemize}
\item[(i)] The surface $S_k = \bigcup_{j=0}^R S_k^{(j)}$ is isotopic to $\tilde{\Sigma}_k$ if $k$ is sufficiently large.
\item[(ii)] The surface $\tilde{\Sigma}_k$ is obtained from $\Sigma_{q_k}$ by $\gamma_0$-reduction (cf. \cite[Section 3]{msy}). 
\item[(iii)] We have $S_k^{(0)} \cap \big ( \bigcup_{j=1}^R S_k^{(j)} \big ) = \emptyset$. Moreover, $|S_k^{(0)}| \to 0$ as $k \to \infty$.
\item[(iv)] $\mathrm{genus}(\tilde{\Sigma}_k)\leq \mathrm{genus}(\Sigma_{q_k})=\tilde h$ by \cite[Inequality (3.2)]{msy}.
\end{itemize}
By assumption, $\Sigma_{q_k}$ is homeomorphic to $N_{\tilde h}$, and $\tilde{\Sigma}_k$ is obtained from $\Sigma_{q_k}$ by $\gamma_0$-reduction. Consequently, one of the connected components of $\tilde{\Sigma}_k$ is an embedded non-orientable surface which, by (iv), must have genus less than or equal to   $\tilde h$  and thus  is homeomorphic to $N_{\tilde h}$. Hence, if $k$ is sufficiently large, then one of the connected components of $S_k$ is homeomorphic to $N_{\tilde h}$. Let us denote this connected component by $E_k$. Since $E_k \in \mathcal{F}$, we have $|E_k| \geq \mathcal{A}(M,g) > 0$. On the other hand, we have $|S_k^{(0)}| \to 0$ as $k \to \infty$. Putting these facts together, we conclude that $|E_k| > |S_k^{(0)}|$ if $k$ is sufficiently large. Hence, if $k$ is sufficiently large, then $E_k$ cannot be contained in $S_k^{(0)}$. Since $E_k \subset S_k$ is connected, it follows that $E_k$ is a connected component of $S_k^{(i)}$ for some integer $i \in \{1,\hdots,R\}$. Hence, $E_k$ is either homeomorphic to $\Sigma^{(i)}$ or to an oriented double cover of $\Sigma^{(i)}$. The last case cannot happen and thus  $\Sigma^{(i)} \in \mathcal{F}$. Moreover, it follows from (\ref{area}) that $|\Sigma^{(i)}| \leq \mathcal{A}(M,g)$. Hence, the surface $\Sigma^{(i)}$ is the desired minimizer.
\end{proof}

In order to state the next result, we adopt the following notation. If $M$ does not admit non-orientable embedded surfaces, we set $\tilde h$ to be the Heegaard genus $h$ of $M$. Otherwise we set $\tilde h$ to be the lowest possible genus among all non-orientable embedded surfaces in $M$, as above.
 
 \begin{cor}\label{consequence}
Let $(M,g)$ be a compact  orientable three-manifold. There  exists an embedded minimal surface $\Sigma \subset M$ with  $\mbox{ind}(\Sigma)\leq 1$ and $g(\Sigma)\leq \tilde h.$
\end{cor}
\begin{proof}
If $M$ admits  non-orientable embedded surfaces the result follows from Proposition \ref{nonorientable}. Therefore  we can assume that every embedded surface of $M$ is orientable. If $M$ admits stable minimal surfaces of genus less than or equal to  $h$ the result follows immediately. The remaining case is when the ambient manifold satisfies the $(\star)_h$-condition, in which case the result follows from Theorem \ref{heegaard.genus}.
\end{proof}

 \begin{cor}
   Assume $M=S^3$. Then there exists an embedded minimal sphere $\Sigma$ in $M$ of index at most one.
   \end{cor}

\begin{rmk}
A consequence of Pitts' work \cite{pitts} is that every $3$-manifold admits an embedded minimal surface. Corollary \ref{consequence} gives some extra geometric information on the minimal surface. In \cite{pitts-rubin}, the authors  claimed, without a proof, an index  and genus estimate from which Corollary \ref{consequence} would follow. A related genus estimate was recently proven in \cite{delellis-genus}.
\end{rmk}

\section{Ricci flow and rigidity results}\label{ricci.flow}

 Let $\Lambda$ be  a saturated set of generalized families of surfaces.

We consider $g(t)$, $t \in [0,T)$, a smooth solution to Ricci flow 
$$\frac{\partial }{\partial t}g(t)=-2\mathrm{Ric}(g(t))$$
with $g(0)=g$. 

 \begin{lemm}\label{lemm.lip}
The function $t \mapsto W(M, \Lambda, g(t))$ is Lipschitz continuous.
\end{lemm}

\begin{proof}
Let $t_0 \in [0,T)$, and choose $C > 0$ such that $\sup_M |\text{\rm Ric}(g(t))|_{g(t)} \leq C$ for all $t \in [0,t_0]$. Hence
\[e^{-2C |t_1 - t_2|} \, g(t_1) \leq g(t_2) \leq e^{2C |t_1 - t_2|} \, g(t_1)\] 
for all  $t_1,t_2 \in [0,t_0]$.

Given $\delta>0$, let  $(\Sigma_s)_{s \in [-1,1]}$ denote a sweepout  such that 
$$\sup_{s\in [-1,1] }\H^2_{g(t_1)}(\Sigma_s) \leq W(M,\Lambda, g(t_1)) +\delta.$$

Now,
\begin{eqnarray*} 
W(M,\Lambda, g(t_2))\leq \sup_{s \in [-1,1]}\H^2_{g(t_2)}(\Sigma_s) &\leq& e^{2C |t_1 - t_2|}\sup_{s\in[-1,1]}\H^2_{g(t_1)}(\Sigma_s)\\
&\leq& e^{2C |t_1 - t_2|}\Big(W(M,\Lambda, g_{t_1})+\delta)\Big).
\end{eqnarray*}

Letting $\delta \rightarrow 0$, and reversing the roles of $t_1$ and $t_2$, we obtain 
$$ e^{-2C |t_1 - t_2|}W(M,\Lambda, g(t_1))\leq W(M,\Lambda, g(t_2))\leq e^{2C |t_1 - t_2|}W(M,\Lambda, g(t_1))$$
for all $t_1,t_2 \in [0,t_0]$. The result follows.
\end{proof}

\begin{prop}\label{width.inequality.1}
  Let $h$ be the Heegaard genus of $M$ and assume $(M,g(t))$ satisfies the $(\star)_h$ -condition for all $0 \leq t < T'$, with $T' \leq T$.  Then
$$W(M,\Lambda^h, g(t))\geq  W(M,\Lambda^h,g)-\Big(16\pi-8\pi\Big[\frac{h}{2}\Big]\Big) t$$
for all $0\leq t<T'$.
\end{prop}

 \begin{proof}
 Suppose the assertion is false. This means that there exists  $\tau \in (0,T')$ such that 
$$ W(M, \Lambda^h,g(\tau)) < W(M,\Lambda^h, g) - \Big(16\pi-8\pi\Big[\frac{h}{2}\Big]\Big) \tau.$$

Let $\varepsilon > 0$ be such that 
$$ W(M,\Lambda^h,g(\tau)) < W(M, \Lambda^h,g) - \Big(16\pi-8\pi\Big[\frac{h}{2}\Big]\Big) \tau-2\varepsilon\tau,$$
and define 
\begin{eqnarray*}
t' = \inf \Big\{ t \in [0,T'): W(M,\Lambda^h,g(t)) &<&  W(M, \Lambda^h,g)\\
&&- \left(16\pi-8\pi\Big[\frac{h}{2}\Big] + \varepsilon \right) t - \varepsilon \tau \Big\}.
\end{eqnarray*} 

Clearly, $t' \in (0,\tau)$. Moreover, we have 
\begin{equation}\label{width.ineq.1}
 W(M, \Lambda^h,g(t')) - W(M, \Lambda^h,g(t)) \leq -\left(16\pi-8\pi\Big[\frac{h}{2}\Big]+ \varepsilon\right) \, (t' - t)
 \end{equation}
for all $t \in [0,t')$.  

Since  $(M,g(t'))$  satisfies the $(\star)_h$ -condition,  we can choose  $(\Sigma_s)_{s\in [-1,1]}$  as the sweepout given by Theorem \ref{heegaard.genus}, with $$|\Sigma_0|=W(M,\Lambda^h,g(t')),$$   and set $f(s,t)=|\Sigma_s|_{g(t)}$.

A standard computation using the Gauss equation  shows that
\begin{align*} 
\frac{\partial f}{\partial t}(0,t') = \frac{d}{dt} |\Sigma_0|_{g(t)}(t') & = -\int_{\Sigma_0} \Big(R - \text{\rm Ric}(\nu,\nu)\Big)\,d\mu\\
& = -4\pi\chi(\Sigma_0) - \int_{\Sigma_0}  \Big(\text{\rm Ric}(\nu,\nu)+|A|^2\Big)\,d\mu, 
\end{align*} 
where all  geometric quantities are computed with respect  to $g(t')$. 

From  Proposition \ref{prop.index.estimate},  part  (iii), we obtain
\begin{eqnarray*}
\frac{\partial f}{\partial t}(0,t') &\geq& 8\pi(h-1)-8\pi\Big( \Big[\frac{h+1}{2}\Big]+1\Big)\\
&=& -16\pi +8\pi\Big[\frac{h}{2}\Big].
\end{eqnarray*} Since  $f$ is smooth in a neighborhood of $(0,t')$,  this implies 
\begin{eqnarray*}
 f(s,t) &\leq& f(s,t') -\left(16\pi -8\pi\Big[\frac{h}{2}\Big]+ \frac{\varepsilon}{2}\right) (t-t')\\
 &\leq& W(M,\Lambda^h,g(t'))-\left(16\pi -8\pi\Big[\frac{h}{2}\Big]+ \frac{\varepsilon}{2}\right) (t-t')
\end{eqnarray*}
for all $(s,t)$ close to $(0,t')$ with $t\leq t'$. Since $s \rightarrow f(s,t')$ has a unique maximum point at $s=0$, we conclude by continuity that
$$
\sup_{s\in [-1,1]} f(s,t) \leq W(M,\Lambda^h,g(t'))-\left(16\pi -8\pi\Big[\frac{h}{2}\Big]+ \frac{\varepsilon}{2}\right) (t-t')
$$
for all $t$ sufficiently close to $t'$. This gives 
$$
W(M,\Lambda^h,g(t)) \leq W(M,\Lambda^h,g(t'))-\left(16\pi -8\pi\Big[\frac{h}{2}\Big] + \frac{\varepsilon}{2}\right) (t-t')
$$
for such $t$, which is in contradiction with inequality (\ref{width.ineq.1}). This finishes the proof of the proposition.
\end{proof}

 \begin{cor}\label{width.inequality.2}
 Suppose $(M,g)$ has positive Ricci curvature, and $M$ contains no non-orientable embedded  surface. Let $h$ be the Heegaard genus of $M$. Then
$$W(M,\Lambda^h, g(t))\geq  W(M,\Lambda^h,g)-\Big(16\pi-8\pi\Big[\frac{h}{2}\Big]\Big) t$$
for all $0\leq t<T$.
\end{cor}

 \begin{proof}
 Since positive Ricci curvature is preserved by Ricci flow in dimension three (\cite{hamilton}), we have that $(M,g(t))$ contains no stable embedded minimal surface for all $0\leq t<T$ and thus must satisfy the $(\star)_h$ -condition. The corollary follows immediately from Proposition \ref{width.inequality.1}.
 \end{proof}  
 
 We denote the scalar curvature of $(M,g)$ by $R$.

\begin{thm} \label{thm.pos.ricci} Suppose $(M,g)$ has positive Ricci curvature, and $M$ contains no non-orientable embedded surface. Let $h$ be the Heegaard genus of $M$.  If  $R\geq 6$, then $$W(M,\Lambda^h, g)\leq 4\pi-2\pi\left[\frac{h}{2}\right] \leq 4\pi.$$
Moreover, $W(M,\Lambda^h, g)= 4\pi$ if and only if $g$  has constant sectional curvature one and $M=S^3$. 
\end{thm}
\begin{proof}
 Let $(g(t))_{0\leq t<T}$ denote a maximal solution of Ricci flow, with $g(0)=g$.  From Corollary \ref{width.inequality.2} we obtain 
 \begin{equation}\label{width.equation}
 W(M,\Lambda^h,g(t))\geq W(M,\Lambda^h,g(0))-\left(16\pi-8\pi\left[\frac{h}{2}\right]\right) t.
 \end{equation}

We now argue that
\begin{equation}\label{zero}
\lim_{t\to T} W(M,\Lambda^h,g(t))=0.
\end{equation}
This follows because, from Proposition \ref{prop.index.estimate}, part  (iii), and Theorem \ref{heegaard.genus} ($(M,g)$ satisfies the $(\star)_h$-condition), we have
$$\min_{M}R(g(t))W(M,\Lambda^h,g(t))\leq 24\pi+16\pi\left(\frac{h}{2}-\left[\frac{h}{2}\right]\right),$$ 
 and according to Theorem 15.1 of \cite{hamilton} we have
$$\lim_{t\to T} \min_{M}R(g(t))=+\infty.$$
Combining \eqref{zero} with \eqref{width.equation} we obtain
\begin{equation}\label{1/4}
W(M,\Lambda^h,g)\leq \left(16\pi-8\pi\left[\frac{h}{2}\right]\right)T.
\end{equation}

 From the evolution equation of the scalar curvature:
$$\frac{\partial}{\partial t} R(g_t) = \Delta R(g_t) + \frac{2}{3} \, R(g_t)^2 + 2 \, |\mathring {\mathrm{Ric}}|^2,$$
we get
$$\frac{\partial}{\partial t} R(g_t) \geq \Delta R(g_t) + \frac{2}{3} \, R(g_t)^2.$$
Here $\mathring {\mathrm{Ric}} = \mathrm{Ric}-Rg/3$.
If $\min_{M}R(g(t_1))=k_1$, the Maximum Principle tells us that
\begin{equation}\label{R}
\min_M R(g(t))\geq \frac{3k_1}{3-2k_1(t-t_1)}\quad\mbox {for all  } t_1\leq t<T.
\end{equation}
Using the inequality above with $t_1=0$ and $k_1=6$ we obtain
\begin{equation}\label{lower.r} \min_{M}R(g(t))\geq \frac{6}{1-4t}.
\end{equation}
Hence  $T\leq 1/4$ and it follows from inequality \eqref{1/4}  that $$W(M,\Lambda^h,g) \leq 4\pi-2\pi\left[\frac{h}{2}\right].$$

If $W(M,\Lambda^h,g)=4\pi$, then  we must have $T=1/4$ and $h=0$ or $1$. We  first show that  $g$ must be Einstein, therefore of constant curvature.  The expression on the right hand side of \eqref{R} must be finite for all $t_1\leq t<1/4$,  hence
$3-2k_1(t-t_1)>0$ for all $t_1\leq t<1/4 $ and so
 $$ \min_{M}R(g(t_1))\leq \frac{6}{1-4t_1}$$
 for all $0\leq t_1<1/4$. This implies   equality  in \eqref{lower.r} and the Maximum Principle implies that  $g$  must be Einstein.
 
 To conclude that $M=S^3$ it is enough to show that $h=0$ because in that case $M$ contains a minimal embedded sphere and  Frankel's Theorem \cite{frankel} implies $M$ is simply connected. If $h=1$ then $M$ contains a minimal embedded torus $T$ which realizes the width and so, by Theorem \ref{heegaard.genus},  any other embedded minimal  torus must have area bigger than $|T|=4\pi$. It is a classical fact that manifolds  with Heegaard genus one are either Lens spaces $L(p,q)$ or $S^2\times S^1$ (see \cite[Section 8.3.4]{stillwell} or \cite[Theorem 1.6]{sav}). Thus $M$ contains a flat torus of area  $2\pi^2/p<4\pi$ (projection of Clifford torus) which is a contradiction. 
\end{proof}

If $M$ contains   non-orientable embedded surfaces, we can consider the invariant  $\mathcal{A}(M,g)$ defined in  Section \ref{section.minimal.index}.

\begin{lemm}\label{lemm.lip2}
The function $f(t)=\mathcal{A}(M,g(t))$ is Lipschitz continuous.
\end{lemm}
\begin{proof}
The proof is analogous to the proof of  Lemma \ref{lemm.lip}.
\end{proof}
\begin{prop}\label{prop.no.orientable}
For all $0\leq t< T$ we have
$$\mathcal{A}(M,g(t))\geq  \mathcal{A}(M,g)-8\pi t.$$
\end{prop}
\begin{proof}
Suppose that $\Sigma \in \mathcal{F}$ is such that $|\Sigma|_{g(t_0)}=\mathcal{A}(M,g(t_0))$.  From Proposition \ref{prop.index.estimate}, part  (ii), we obtain
\begin{eqnarray*} 
\frac{d}{dt} |\Sigma|_{g(t)}(t_0)   &=& -\int_{\Sigma} \Big(R - \text{\rm Ric}(\nu,\nu)\Big)\,d\mu \\
&=& -2\int_{\Sigma}K\,d\mu - \int_{\Sigma} \Big( \text{\rm Ric}(\nu,\nu)+|A|^2\Big)\,d\mu\\
&=&-2\pi\chi(\tilde\Sigma)-\int_{\Sigma}\Big(  \text{\rm Ric}(\nu,\nu)+|A|^2\Big)\,d\mu\\
 &\geq& 4\pi(g(\tilde \Sigma)-1)-4\pi(g(\tilde \Sigma)+1)\\
 &=&-8\pi.
\end{eqnarray*} 
Using this calculation and Proposition \ref{nonorientable}, we can argue exactly like in Proposition \ref{width.inequality.1} or in \cite[Proposition 10]{bben} to achieve the desired result.
\end{proof}

\begin{thm}\label{A.inequality} Assume $(M,g)$ has positive Ricci curvature and contains  embedded non-orientable surfaces.  If $R\geq 6$, then $\mathcal{A}(M, g)\leq 2\pi.$ 
\end{thm}
\begin{proof}
Let $(g(t))_{0\leq t<T}$ denote a maximal solution of Ricci flow with $g(0)=g$. From Proposition \ref{prop.no.orientable},
$$\mathcal{A}(M, g(t))\geq \mathcal{A}(M, g)-8\pi t.$$
Reasoning like in the proof of Theorem \ref{thm.pos.ricci} we have
$$\lim_{t\to T}\mathcal{A}(M, g_t)=0.$$
Therefore $\mathcal{A}(M, g) \leq 8\pi T$. But we know from \eqref{lower.r} that $T\leq 1/4$, hence $\mathcal{A}(M, g) \leq 2\pi.$ 
\end{proof}

\begin{rmk} 
Notice that the estimates for the width (Theorem \ref{thm.pos.ricci}) and for the $\mathcal{A}$-invariant (Theorem \ref{A.inequality}), both proved using the Ricci flow,  are better than the basic area estimates for index one and stable minimal surfaces obtained through the Hersch's trick (Appendix \ref{area.estimates}).
\end{rmk}

Let $\mathcal{J}$ be the collection of all embedded minimal surfaces $\Sigma \subset M$ with ${\rm ind}(\Sigma) \leq 1$.

\begin{thm} \label{ricci.rigidity} Suppose  $M$ has positive Ricci curvature and $R\geq 6$. Then there exists an embedded minimal surface $\Sigma$, with ${\rm ind}(\Sigma) \leq 1$, such that
$$
|\Sigma| \leq 4\pi.
$$
Moreover, we have that
$$\inf_{\Sigma \in \mathcal{J}} |\Sigma|= 4\pi$$
if and only if  $g$ has constant sectional curvature one and $M=S^3$.
\end{thm}

\begin{proof}
Suppose $M$ contains non-orientable embedded surfaces. It follows from Proposition \ref{nonorientable} and Theorem \ref{A.inequality} that there exists an embedded  minimal surface $\Sigma \in \mathcal{F}$, with ${\rm ind}(\Sigma)=0$, and $|\Sigma|\leq 2\pi$.

Suppose now that $M$ does not contain non-orientable embedded surfaces, and let $h$ be the Heegaard genus of $M$. Then, $(M,g)$ satisfies the $(\star)_h$-condition  and so, by Theorem \ref{heegaard.genus}, we have the existence of  an embedded minimal surface $\Sigma_2 \subset M$ with ${\rm ind}(\Sigma_2)=1$ and such that $|\Sigma_2| = W(M,\Lambda^h,g)$. Theorem \ref{thm.pos.ricci} implies
$|\Sigma_2| \leq 4\pi$. 

If $\inf_{\Sigma \in \mathcal{J}} |\Sigma|= 4\pi$, then it follows from the previous arguments that $M$ does not contain non-orientable embedded surfaces and $W(M,\Lambda^h,g)=4\pi$. Hence, by Theorem \ref{thm.pos.ricci}, $g$ has constant sectional curvature one and $M=S^3$.
\end{proof}

\medskip 

\begin{rmk} We point out that the surface $\Sigma$ constructed in Theorem \ref{ricci.rigidity} can be chosen to satisfy the genus bound:  $g(\Sigma)\leq \tilde h$. 
\end{rmk}


\medskip

We now turn to the case in which $M$ is diffeomorphic to the $3$-sphere $S^3$, whose Heegaard genus is zero. In that case we can take $\Lambda^0$ to be the smallest saturated set that contains the family $\{\Sigma_t\}$ of   level sets of the height function $x_4:S^3 \subset \mathbb{R}^4 \rightarrow \mathbb{R}$. We define the {\it width} of $(S^3,g)$ to be 
 $$W(S^3,g)=W(S^3,\Lambda^0,g).$$
 
\begin{thm}\label{thm.sphere} Assume $(S^3,g)$ has no stable embedded minimal  spheres.  If $R \geq 6$, there exists an embedded minimal sphere
$\Sigma$, of index one, such that
$$W(S^3,g)=|\Sigma|=\inf_{S \in \mathcal{E}_{0} }|S| \leq 4\pi.$$
The  equality $W(S^3,g)=4\pi$ holds if and only if $g$  has constant sectional curvature one.
\end{thm}
\begin{rmk}
 Theorem \ref{mainthm.1} is an immediate consequence of Theorem \ref{thm.sphere}.
\end{rmk}
\begin{proof}
 $(S^3,g)$ satisfies the $(\star)_0$ -condition because the three-sphere  contains no non-orientable embedded surface. Hence Theorem \ref{heegaard.genus} implies the existence of an index one embedded minimal sphere $\Sigma$ with $|\Sigma|=\inf_{S \in \mathcal{E}_{0} }|S|=W(S^3,g)$. From  Proposition \ref{prop.index.estimate}, part (iii), we have $|\Sigma| \leq 4\pi$.

Suppose $|\Sigma|=4\pi$.   To show that $g$ is Einstein we argue essentially like in Theorem \ref{thm.pos.ricci}. 

Let $g(t)$, $t \in [0,\varepsilon)$, be a  solution of Ricci flow with $g(0)=g$. The Maximum Principle applied to the evolution equation of the scalar curvature  implies 
$$
\min_{M}R(g(t))\geq \frac{6}{1-4t}.
$$
 It  follows from  Proposition \ref{prop.index.estimate}, part  (i),  that any   stable embedded minimal surface in  $(M,g(0))$ would have to be a sphere with  area at most  $\frac{4\pi}{3}$. Hence $(M,g(0))$ does not contain stable embedded minimal surfaces and so, by  Proposition \ref{prop.no.stable} in the Appendix,  $(M,g(t))$ contains no stable  embedded minimal surface provided we choose $\varepsilon$ small enough.  Therefore Proposition \ref{width.inequality.1} implies that $W(M, g(t))\geq 4\pi(1-4t)$.  We know from Theorem \ref{heegaard.genus} that   $W(M, g(t))$ is the area of an index one embedded minimal sphere in $(M,g(t))$ and thus
   Proposition \ref{prop.index.estimate}, part  (iii), implies  that 
 $$\min_{M}R(g(t))\leq \frac{6}{1-4t}.$$
 Therefore $$\min_{M}R(g(t))= \frac{6}{1-4t}$$ and the maximum principle tells us that $g$ is Einstein. Since the dimension is three, $g$ has  constant sectional curvature one.
\end{proof}

 \begin{thm}\label{thm.sphere.scalar.curvature}
Let $g$ be a metric  on $S^3$ with scalar curvature $R \geq 6$. If $g$ does not have constant sectional curvature one, then there exists an embedded minimal sphere $\Sigma$, of index zero or one, with  $|\Sigma| < 4\pi$.
\end{thm}

\begin{proof}
If $(S^3,g)$ contains a stable embedded minimal sphere $\Sigma$, then $|\Sigma| \leq \frac{4\pi}{3}$ by Proposition \ref{prop.index.estimate},  part (i). If not, $W(M) < 4\pi$ by 
Theorem \ref{thm.sphere} and  the corollary  follows from Theorem \ref{heegaard.genus}.
\end{proof}

\appendix

\section{Area estimates}\label{area.estimates}

If $\Sigma$ is non-orientable, we denote by $\tilde \Sigma$ the orientable double cover of $\Sigma$.

The next proposition collects some basic estimates  for minimal surfaces of index less than or equal to  one. These estimates are  based on the so called Hersch's trick which has been used, among other places, in \cite{bben, bbn, Christodoulou-Yau, Colding-Minicozzi1,  rr, ross,  yy, yau}.
 
 \begin{prop}\label{prop.index.estimate} Assume the scalar curvature of $M$ satisfies $R\geq k_0>0.$ Let $\Sigma$ be an embedded minimal surface of genus $g(\Sigma)$.
 \begin{itemize}
 \item[(i)] If $\Sigma$ is stable  and orientable then it is a sphere with $k_0|\Sigma|\leq 8\pi$. Equality implies that $R=k_0$ on $\Sigma$.
 \item[(ii)] If $\Sigma$ is  stable and non-orientable  then
 $$\int_\Sigma \text{\rm Ric}(\nu,\nu) + |A|^2 \, d\mu\leq 4\pi(g(\tilde \Sigma)+1)$$
and $$k_0|\Sigma|\leq  12\pi+4\pi g(\tilde \Sigma).$$
 \item[(iii)] If $\Sigma$ is index one and orientable then
$$ \int_\Sigma \text{\rm Ric}(\nu,\nu) + |A|^2 \, d\mu\leq 8\pi\left(\left[\frac{g(\Sigma)+1}{2}\right]+1\right)$$ and
$$k_0|\Sigma|\leq 24\pi+16\pi\left(\frac{g(\Sigma)}{2}-\left[\frac{g(\Sigma)}{2}\right]\right),$$
where $[x]$ denotes the integer part of $x$.
 \end{itemize}
 \end{prop} 
 
 \begin{proof}
  The  statement (i) was proven in \cite{bbn} (see identity (4)).
  
  The first inequality in (ii) follows  from Lemma 2 and identity (2') of \cite{ross}. The second inequality is a consequence of the Gauss equation: 
  \begin{eqnarray*}
  \frac{k_0}{2}|\Sigma|\leq \int_{\Sigma}\frac{R}{2}d\mu&=&\int_{\Sigma}\Big(\textrm{Ric}(\nu,\nu)+\frac{|A|^2}{2}+K \Big)d\mu\\
  &\leq& \int_{\Sigma}\Big(\textrm{Ric}(\nu,\nu)+{|A|^2} \Big)d\mu+2\pi(1-g(\tilde \Sigma))\\
  &\leq& 6\pi+2\pi g(\tilde \Sigma).
  \end{eqnarray*}

 If $\Sigma$ is orientable of index one, then
 $$ \int_\Sigma \Big(\text{\rm Ric}(\nu,\nu) + |A|^2\Big) \, d\mu\leq 8\pi \textrm{deg}(\phi),$$
where $\phi:\Sigma\longrightarrow S^2$ is a conformal map (see page 127 of \cite{yau}). Since we can choose $\phi$ satisfying $$\textrm{deg}(\phi)\leq \left[\frac{g(\Sigma)+1}{2}\right]+1,$$
as in page 299 of \cite{rr},  the first inequality in (iii) follows. The last inequality again follows from the Gauss equation:
\begin{eqnarray*}
  \frac{k_0}{2}|\Sigma|\leq \int_{\Sigma}\frac{R}{2}d\mu&\leq&\int_{\Sigma}\Big(\textrm{Ric}(\nu,\nu)+\frac{|A|^2}{2}+K \Big)d\mu\\
  &\leq& 8\pi\left(\left[\frac{g(\Sigma)+1}{2}\right]+1\right)+4\pi(1-g(\Sigma))\\
  &=&12\pi+{8\pi}\left(\frac{g( \Sigma)}{2}-\left[\frac{g( \Sigma)}{2}\right]\right).
  \end{eqnarray*}
   \end{proof}

\section{Compactness}
For completeness we include the proof of a well known result among the specialists.

 \begin{prop}\label{prop.no.stable} Suppose $(M,g)$ is  a compact Riemannian 3-manifold that contains   no stable embedded minimal surfaces. Given a constant $C>0$,  there exists a $C^{3,\alpha}$ neighborhood $\mathcal{U}$ of $g$ so that every metric $g'$ in $\mathcal{U}$ contains  no stable embedded minimal surface of  area smaller than  $C$.
 \end{prop} 
 \begin{proof}
 We start by arguing that $M$ contains no non-orientable embedded  surface.  From \cite[Lemma 1]{msy} we have that non-orientable  surfaces can not have arbitrarily small area and thus we can minimize area in their isotopy class \cite[Theorem 1]{msy} to obtain a stable embedded minimal surface, a contradiction.

Assume $g_i$ is a sequence of metrics converging  in $C^{3,\alpha}$ to $g$ with $\Sigma_i$ a sequence of connected  stable embedded minimal surfaces in $(M,g_i)$ with area smaller than $C$. From Schoen's curvature estimate for stable surfaces \cite[Theorem 3]{schoen} we have that the second fundamental form of $\Sigma_i$ is uniformly bounded and so, because their area is bounded, we obtain a uniform genus bound for all $\Sigma_i$. Standard compactness arguments imply that, after passing to a subsequence, $\Sigma_i$ converges to a smooth minimal surface $\Sigma$ (with multiplicity $m$) which is embedded by the maximum principle. Furthermore (for instance \cite[Lemma 1]{white} or \cite[Theorem 4.2.1]{perez}), there is $r_0$ small so that, for all $i$ sufficiently large and all $x\in \Sigma$, we have
\begin{equation}\label{graphical}
\Sigma_i\cap B_{r_0/2}(x)\subset\bigcup_{j=1}^m\left\{\exp_y(f^i_j(y)\nu(y))\,|\, y\in \Sigma\cap B_{r_0}(x)\right\}\subset \Sigma_i\cap B_{2r_0}(x),
\end{equation}
where $\nu$ is a unit normal vector field in $\Sigma$ computed with respect to $g$, $\exp_x$ is the exponential map at $x$ for the metric $g$, and the $C^{2,\alpha}$-norm of each $f^i_j$  tends to zero when $i$ tends to infinity. 

 Consider the continuous sections $X_i, Y_i$ of the normal bundle of $\Sigma$ given by $$X^i(x)=\max_{j\in\{1,\ldots,n\}}\{f_j^i(x)\}\nu(x), \quad Y^i(x)=\min_{j\in\{1,\ldots,n\}}\{f_j^i(x)\}\nu(x).$$

The fact that $\Sigma_i$ is embedded implies that, for all $i$ sufficiently large,  there is $\varepsilon_i>0$ so that $|f_i^j(x)-f^k_j(x)|>\varepsilon_i$ for all $x\in \Sigma$ and all $j\neq k$ in $\{1,\cdots,m\}$ which means $X_i, Y_i$ are smooths sections of the normal bundle. Because this bundle is trivial there are smooth functions $a_i, b_i$ defined on  $\Sigma$ so that  $X_i=a_i\nu, Y_i=b_i\nu$.  Thus, from \eqref{graphical}, we obtain that $m=1$   because otherwise $\Sigma_i$ would have at least two distinct connected components.  Therefore   $\Sigma_i$ converges in $C^{2,\alpha}$ to $\Sigma$ and so $\Sigma$ must be stable as well. This proves the lemma.
 \end{proof}

\bibliographystyle{amsbook}

\end{document}